\documentclass[english,12pt]{amsart}
\usepackage{babel}
\usepackage{amstext}
\usepackage{amsmath}
\usepackage{amsfonts}
\usepackage{latexsym}
\usepackage{ifthen}
\usepackage{xypic}
\xyoption{all}
\newcommand{\id}{{\rm id}}

\newcommand{\nonet}{{non\;et}}

\newcommand{\Jac}{{\rm Jac}}
\newcommand{\reg}{{\rm reg}}

\newcommand\sE{{\mathcal E}}

\newcommand\sO{{\mathcal O}}

\newcommand\bZ{{\mathbb Z}}
\newcommand\bC{{\mathbb C}}
\newcommand\bQ{{\mathbb Q}}

\newcommand\bP{{\mathbb P}}

\newcounter{lemma}

\newtheorem{lemma1}[lemma]{\setcounter{equation}{0}}

\newenvironment{lemma}{\begin{lemma1}{\bf Lemma.}}{\end{lemma1}}

\newenvironment{example}{\begin{lemma1}{\bf Example.}\rm}{\end{lemma1}}

\newenvironment{theorem}{\begin{lemma1}{\bf Theorem.}}{\end{lemma1}}

\newenvironment{proposition}{\begin{lemma1}{\bf Proposition.}}{\end{lemma1}}

\newenvironment{corollary}{\begin{lemma1}{\bf Corollary.}}{\end{lemma1}}

\newenvironment{remark}{\begin{lemma1}{\bf Remark.}\rm}{\end{lemma1}}

\newenvironment{Induction Step}{\begin{lemma1}{\bf Induction Step.}}{\end{lemma1}}
\newenvironment{Proof of Theorem 1.2}{\begin{lemma1}{\bf Proof of Theorem 1.2.}}{\end{lemma1}}

\def\srelbar{\vrule width0.6ex height0.65ex depth-0.55ex}
\def\merto{\mathrel{\srelbar\kern1.3pt\srelbar\kern1.3pt\srelbar
    \kern1.3pt\srelbar\kern-1ex\raise0.28ex\hbox{${\scriptscriptstyle>}$}}}
\def\build#1^#2_#3{\mathrel{\mathop{\null#1}\limits^{#2}_{#3}}}
\def\ssbt{\,{\scriptscriptstyle\bullet}\,}

\def\noop#1{}

\title {Compact Manifolds covered by a torus}
\author{Jean-Pierre Demailly, Jun-Muk Hwang and Thomas Peternell}

\address{\rm
Jean-Pierre Demailly, Institut Fourier, Universit\'e de Grenoble I,
France\vskip0pt
\emph{e-mail}\/: {\tt demailly@fourier.ujf-grenoble.fr}
\vskip2pt
Jun-Muk Hwang, KIAS, School of Mathematics, Seoul, Korea
\vskip0pt
work supported by the  SRC Program of Korea Science and
Engineering Foundation,
\vskip0pt
KOSEF grant funded by the Korea 
government (MOST) No.\ R11-2007-035-02001-0
\vskip0pt
\emph{e-mail}\/: {\tt jmhwang@kias.re.kr}
\vskip2pt
Thomas Peternell, Mathematisches Institut, Universit\"at Bayreuth,
Germany
\vskip0pt
\emph{e-mail}\/: {\tt thomas.peternell@uni-bayreuth.de}
}

\date{\today}

\begin{document}

\maketitle

\vskip.5cm
\emph{\hfill dedicated to Gennadi Henkin}
\vskip0.5cm
$~$
\begin{abstract}
Let $X$ be a compact complex manifold which is the image of a complex
torus by a holomorphic surjective map  $A \to X$. We prove that $X$ is
K\"ahler and that up to a finite \'etale cover, $X$ is a product of
projective spaces by a torus.
{\vskip4pt
\noindent
{\sc R\'esum\'e.} Soit $X$ une vari\'et\'e analytique complexe
compacte qui est l'image d'un tore complexe $A$ par une application
holomorphe surjective $A \to X$. Nous montrons que $X$ est
k\"ahl\'erienne et que modulo un rev\^etement \'etale fini, $X$ est un
produit d'espaces projectifs complexes par un tore.
\vskip4pt
\noindent
{\sc Zusammenfassung.} Sei $X$ eine kompakte komplexe
Mannigfaltigkeit, die Bild eines komplexen Torus $A$ unter einer
surjektiven holomorphen Abbildung $A \to X$ ist. Wir zeigen, dass $X$
eine K\"ahlermannigfaltigkeit ist und dass $X$ bis auf endliche
\'etale \"Uberlagerung das Produkt von projektiven R\"aumen und eines
Torus ist.}
\end{abstract}

$~$
\tableofcontents
\vskip.5cm

\vskip.5cm

\noop{\noindent
{\sc Mots-cl\'es.} Tore complexe, vari\'et\'e ab\'elienne, espace
projectif, vari\'et\'e k\"ahl\'erienne, morphisme d'Albanese, groupe
fondamental, rev\^etement \'etale, diviseur de ramification, diviseur
nef, fibr\'e tangent nef, fibr\'e anti-canonique, fibr\'e vectoriel
num\'eri\-quement plat.  \vskip.5cm}

\vskip1cm

\section{Introduction}
Consider a compact complex manifold $X$ which is the image
of a complex torus $A$ by a surjective map
$$
f: A \to X.
$$
The purpose of this note is a classification of $X$ up to finite
\'etale cover. When $A$ is a simple abelian variety and $f$ is not an
isomorphism, Debarre [Deb89] has shown that $X \simeq \bP_n$. In the
case of a general abelian variety~$A$, Hwang-Mok [HM01] proved that
$X$ is a tower of bundles
$$
X=X_0\to X_1\to\ldots\to X_k=Y
$$
over a base  $Y$ which is an unramified quotient of an abelian variety,
the fibers of which are projective spaces $\bP_{n_j}$. Our result extends
this structure theorem to the
case of general (non necessarily projective) tori, and clarifies the
structure of the tower of projectives bundles by proving
that it is in fact a locally trivial bundle whose fibers are products
of projective spaces.

\begin{theorem} Let $X$ be a connected compact complex manifold which
is the image of a complex torus $A$ by a surjective
holomorphic map $ f: A \to X$.
Then
\begin{enumerate}
\smallskip
\item $X$ is K\"ahler.
\smallskip
\item There exists a finite \'etale cover $X'$ of $X$ such that
$$
X'\simeq \bP_{n_1}\times\ldots\times\bP_{n_k}\times B
$$
is a product of projective spaces with the Albanese torus
$B=A(X')$. Moreover $f$ admits a lifting $f':A'\to X'$
where $A'$ is an \'etale cover of $A$ and
$$
f':A'=A_1\times \ldots\times A_k\times B\to
\bP_{n_1}\times\ldots\times\bP_{n_k}\times B
$$
is given by a product of factors $f_{n_j}:A_j\to\bP_{n_j}$ and
$\id:B\to B$, for suitable abelian varieties $A_j$
and a $($non necessarily algebraic$)$ torus $B$.

\smallskip
\item $X$ is an \'etale quotient of the bundle $X'=
\bP_{n_1}\times\ldots\times\bP_{n_k}\times B\to B$, in other words
$X$ is a $\bP_{n_1}\times\ldots\times\bP_{n_k}$-bundle over an
\'etale quotient $M$ of the torus~$B$, such that the trivial
fibration $X'\to B$ is the pull-back of $X\to M$ by the
quotient map~$B\to M$.
\end{enumerate}
\end{theorem}

We use in a very essential way results and methods borrowed from the
paper [DPS94] -- the main goal of which was to study the structure of
compact K\"ahler manifolds with nef tangent bundles -- although it is
unclear whether $X$ should a priori have a nef tangent bundle in the
present setting$\,$; certainly $T_X$ is nef on all curves not
contained in the ramification locus, but we do not have control over
the curves in this locus. On the other hand, one can pull-back
divisors to $A$ and obtain, as a simple consequence, that every
pseudo-effective divisor on $X$ must be nef. This is also a
significant property of manifolds with nef tangent bundles.

The fact that $X$ is K\"ahler is a direct consequence of the fact that any
surjective map $f:A\to X$ as above must be equidimensional -- this is
true even when $X$ is taken in the category of reduced complex spaces.
We then rely on an observation due to Varouchas [Var84], [Var89],
stating that the K\"ahler property is preserved by proper equidimensional
morphisms under a normality assumption. This is discussed in more detail
in the appendix, where we present a simple approach of this result.

\section{Elementary reductions}
\setcounter{lemma}{0}

First notice that every $n$-dimensional projective subvariety
$Y\subset \bP_N$ can be mapped by a finite morphism $Y\to\bP_n$ onto
projective space by taking a generic linear projection
$\bP_N\merto\bP_n$, and in particular $\bP_n$ can be obtained as a
finite surjective image of an abelian variety. A very elementary
construction consists in writing $\bP_n$ as the symmetric product
$(\bP_1)^{(n)}\,$; then by taking $2:1$ covers $E_j\to\bP_1$ by elliptic
curves, one realizes $\bP^n$ as the image of $E_1\times\ldots\times E_n$
by a $2^nn!:1$ finite (ramified) morphism. Therefore any product
$A_0\times \bP_{n_1}\times \ldots\times\bP_{n_k}$ of a  torus $A_0$ by
projective spaces can be written as a finite surjective image of an
abelian variety~$A$.

Conversely, let $f:A\to X$ be a surjective holomorphic map from a
complex torus onto a reduced complex space $X$ (we do not require $X$
to be smooth in this section).  Then the generic (smooth) fiber
$F_x=f^{-1}(x)$ has a trivial normal bundle, and therefore also a
trivial tangent bundle.  It follows that the connected components of
$F_x$ are translates of subtori. Since subtori form a discrete family,
by looking at the Stein factorization $A\to W\to X$ which has
precisely the connected components of the $F_x$ as fibers of $A\to W$,
we see that there is a subtorus $S\subset A$ and a factorization
$g:A/S\to X$ of $f$ such that $g$ is generically finite.

\begin{lemma} The quotient map $g:A/S\to X$ is finite.
\end{lemma}

\begin{proof} Otherwise, there would exist a positive dimensional component
$Y$ in some fiber $g^{-1}(x)\subset A/S$. Then a small translate $Y+\delta$
would map to a compact irreducible subset contained in a small Stein
neighborhood $V$ of $x$, and the image $g(Y+\delta)$ would be a single
point $y\in V$. This would imply that the generic fiber of
$g:A/S\to X$ of $g$ is positive dimensional, a contradiction.
\end{proof}

\begin{corollary} Every surjective holomorphic map $f:A\to X$ from a torus
$A$ onto a complex space $X$ factorizes as a finite map
$g:A/S\to X$ with respect to some subtorus $S$ of $A$.
\end{corollary}

As a consequence, we may assume from the very beginning that the surjective
map $f:A\to X$ is finite.

\section{The K\"ahler property}
\setcounter{lemma}{0}

In order to conclude that our complex manifold (or complex space) $X$
is K\"ahler, we rely on the
following general statement. Recall that a complex space $Z$ is said to
be K\"ahler if there exists a (smooth) K\"ahler metric $\omega$ on the
regular part $Z_\reg$, which can be extended locally as a K\"ahler
metric in a smooth ambient space, near every singular point of~$Z$.

\begin{theorem} Let $f:Y \to X $ be a proper and surjective holomorphic
map between complex spaces of pure dimensions. Assume that $Y$ is
K\"ahler, that $f$ has equidimensional fibers and that $X$ is reduced
and normal. Then $X$ is K\"ahler.
\end{theorem}

\begin{proof} When $X$ and $Y$ are manifolds, the result is proved in
J.~Varouchas' PhD Thesis [Var84] -- this is in fact sufficient for our 
purposes, and in fact we even only need the case when $f$ is finite$\,$; 
the general case of complex spaces with $X$ normal is due to J.~Varouchas 
[Var89]. 
Since the proofs given by Varouchas are long and technically quite hard,
we outline here a very short proof for the case of finite maps between
manifolds (and refer the reader to the appendix for a more general
situation where the proof remains quite easy).
Let $\omega$ be a K\"ahler metric on $Y$ and $\alpha$ a (smooth)
positive definite hermitian metric on $X$. Then
for every compact set $K\subset X$ we have $f^*\alpha\le C\omega$ on
$f^{-1}(K)$ for some constant $C_K>0$, hence
$$
f_*\omega \ge C_K^{-1}f_*f^*\alpha \ge C_K^{-1}(\deg f)\,\alpha\quad
\hbox{on $K$}.
$$
This implies that $T=f_*\omega$ is a K\"ahler current (see e.g.\
[DP04, 0.5]), i.e.\ that it is bounded below by a smooth positive
definite form. However, a local potential of $T$ can be written as
$\psi(x)=\sum \varphi_j(y_j)$ where $f^{-1}(x)=\{y_j\}$ and
$\varphi_j$ is a local potential of $\omega$ near $y_j$. This implies
immediately that $\psi$ is continuous. In fact the problem occurs only in a
neighborhood of ramification points, and there the continuity follows
immediately from the fact that $f$ is finite. Therefore $\psi$ is
a continuous strictly plurisubharmonic function. Then it follows by
using Richberg's result [Ric68] (see also [Dem82], [Dem92]) that $\psi$ and $T$
can be regularized so as to produce a K\"ahler metric on~$X$.
\end{proof}

For people only interested in pure algebraic geometry, we present
a simpler statement which suffices in the algebraic context.

\begin{theorem} Let $f: A \to X$ be a finite surjective map from an abelian variety to a complex manifold $X.$ Then $X$ is projective.
\end{theorem}

\begin{proof} It immediately follows from the assumption that $X$ is
Moishezon. Let $$ \pi: \hat X \to X$$
a bimeromorphic holomorphic map from a projective manifold $\hat X.$ Choose an ample line bundle $\hat L$ on $\hat X$ and set
$$ L = (\pi_*(\hat L))^{**}.$$
Then $L$ is a big line bundle on $X$. Hence $f^*(L)$ is a big line
bundle on $A$ and therefore ample by [BL04, Proposition 4.5.2].
Hence $L$ is ample.
\end{proof}

\section{Preliminary structure results}
\setcounter{lemma}{0}

Let $X$ be a compact manifold covered by a torus, and let $n = \dim X$
be its dimension. By the above, $X$ is
K\"ahler and possesses a finite ramified covering $f: A \to X $ by a
torus. Let us denote by $R=\{\Jac(f)=0\}$ the ramification divisor
of $f$ in $A$, so that
$$
K_A=0=f^*K_X+R.
$$
We also consider the irregularity $q(X)=h^0(X,\Omega^1_X)$ of $X$ and
the Albanese map
$$
\alpha: X \to A(X),\qquad\dim A(X)=q(X),
$$
and define
$$
\tilde q(X) = \max q(\tilde X),
$$
where $\tilde X \to X$ runs over all finite \'etale covers $\tilde X \to X$,
to be the maximal irregularity of these covers$\,$; the
following result shows in particular that $\tilde q(X)$ is always finite in
the present context.

\begin{proposition}  \begin{enumerate}
\item Every pseudo-effective divisor on $X$ is nef. In particular $-K_X$ is nef.
\smallskip
\item Every big divisor on $X$ is ample.
\smallskip
\item $X$ is Fano iff $-K_X$ is big iff $(-K_X)^n > 0$ iff $R$ is ample.
\smallskip
\item $T_{X\vert C}$ is nef for all curves not contained in the image
$f(R)\subset X$ of the ramification locus.
\smallskip
\item $\pi_1(X)$ is almost abelian, i.e.\ abelian up to a subgroup of
finite index.
\smallskip
\item The Albanese map $\alpha: X \to A(X)$ is surjective.
\smallskip
\item The Albanese map is in fact a submersion with connected fibers.
\end{enumerate}
\end{proposition}

\begin{proof} (1) If $L$ is pseudo-effective, so is $f^*(L)$.  Hence
$f^*(L)$ (which is a line bundle on a torus) is nef, and therefore
so is $L$ by [DPS94, 1.8] or [Pau98, Th.~1]. This applies in
particular to $-K_X$, since $f^*(-K_X) = R$ is nef.
\vskip .2cm  \noindent (2) If $L$ is big, then $X$ is Moishezon and
therefore projective by [Moi67]. As it is well-known, $L$ is then
the sum of an effective $\bQ$-divisor and of an ample $\bQ$-divisor,
and the effective part is nef by (1), hence $L$ must be ample.
\vskip .2cm

\noindent
(3) follows immediately from (1) and (2).
\vskip .2cm

\noindent
(4) is a direct consequence of the sheaf inclusion $\sO(T_A) \subset
f^*\sO(T_X)$ and of the triviality of $T_A$. In fact, if $C$ is not
contained in $f(R)$, the cokernel of the restriction map
$\sO(T_{A|f^{-1}(C)}) \subset f^*\sO(T_{X|C})$ is a torsion sheaf
supported on the finite set~\hbox{$C\cap f(R)$}.\vskip .2cm

\noindent
(5) As $f$ is surjective, the induced morphism
$$ \pi_1(A) \to \pi_1(X)$$
of fundamental groups has an image of finite index in $\pi_1(X)$, bounded by the degree of $f$. \vskip .2cm

\noindent
(6) If $\alpha$ were not be surjective, then the image $\alpha(X)\subset A(X)$ would have a quotient $Y=\alpha(X)/B$ by a subtorus which is a variety of general type (Ueno  [Uen75]). Then there would exist a surjective morphism $A\to Y$
to a variety of general type, which is absurd: the pull-back of a pluricanonical section of $Y$ would yield a section of a certain tensor power $(\Omega^p_A)^{\otimes k}$ possessing zeros, contradicting the triviality of the bundle $\Omega^p_A$.\vskip .2cm

\noindent
(7) The composition
$$
A \build\to^{f}_{}X\build\to^{\alpha}_{}A(X)
$$
is a surjective map between tori, hence is a linear submersion. Its
differential is everywhere surjective, and therefore so is the
differential of $\alpha$. If $\alpha$ has disconnected fibers, we
consider its Stein factorization
$$X \to W \to A(X).$$
The finite map $W\to A(X) $ must be unramified, otherwise $W$  would have positive Kodaira dimension, which is impossible since it is an image of a torus. Therefore $W$ itself is a torus, and must be equal to $A(X)$ by the universality of
the Albanese map. This actually implies that the fibers of $\alpha$ are connected.
\end{proof}

\begin{theorem} Take a finite \'etale cover $\tilde X \to X$ such that
$q(\tilde X)=\tilde q(X)$. Then the Albanese map
$$
\alpha:\tilde X\to A(\tilde X)
$$
is a submersion with connected fibers $F$ which are covered by tori and have
finite fundamental group $\pi_1(F).$
\end{theorem}

\begin{proof} There exists a lifting $\tilde f:\tilde A\to\tilde X$ of $f$ to
some \'etale cover of the torus $A$. Therefore $\tilde X$ is also covered by
a torus and 4.1~(7) implies that $\alpha$ is a submersion with connected
fibers.

Now consider a fiber $F$ of $\alpha$. Then $\tilde f^{-1}(F)$ has a
trivial normal bundle in $\tilde A$, hence $\tilde f^{-1}(F)$ must be
a translate of a subtorus. It follows from this that every fiber $F$ is
also covered by a torus, as well as any finite \'etale cover $\tilde
F$. By what we have just seen, the Albanese map of $\tilde F$ is a
submersion, and we can thus apply Proposition 4.3 below to obtain
$$
\tilde q(X) = \tilde q(F) + \tilde q(A(X)).
$$
Since $\tilde q(A(X))=\dim A(X)=q(X)=\tilde q(X)$, we conclude that
$\tilde q(F) = 0\,$; in other words, $\pi_1(F)$ is finite.
\end{proof}

\begin{proposition} Let $X$ and $Y$ be compact K\"ahler manifolds and $g: X \to Y$ be a surjective submersion with connected fibers. Let $F$ be a fiber
of $g.$ Then
$$
\tilde q(X) \leq \tilde q(F) + \tilde q(Y). \eqno (*).
$$
Assume moreover that $Y$ is a torus, that $\pi_1(F)$ is almost abelian
and that for every finite \'etale cover $\tilde F \to F$, the Albanese
map of $\tilde F\to A(\tilde F)$ is a constant rank map. Then equality
in $(*)$ holds.
\end{proposition}

\begin{proof} This is proposition 3.12 in [DPS94].
\end{proof}

\begin{theorem} If $H^0(X,\Omega^p_X) \ne 0$ for some $p \geq 1,$ then
$\tilde q(X) > 0.$
\end{theorem}

\begin{proof} We may assume $p \geq 2$ and choose a non-zero holomorphic
$p$-form $u$. Let
$$ \Phi: \bigwedge^{p-1}T_X \to \Omega^1_X$$
be the morphism defined by contraction of $(p-1)$-vectors with $u$, and
let $\sE = {\rm Im} \Phi \subset \Omega^1_X$ be the image of $\Phi$.
A priori $\sE $ is a torsion free sheaf of rank $r$. However, if we pull-back
the morphism to $A$ by $f^*$, we get a composition
$$
\bigwedge^{p-1}T_A \to \bigwedge^{p-1}f^*T_X
\build\longrightarrow^{f^*\Phi}_{} f^*\Omega^1_X \to \Omega^1_A
$$
The bundles on both sides are trivial, hence the composition is given by a
constant matrix. This implies that $f^*\Phi$ itself is a bundle morphism of
constant rank and that $f^*\sE$ is a trivial bundle. As a consequence,
$\Phi$  is of constant rank and $\sE$ is a numerically flat vector bundle
on $X$. By [DPS94, Theorem~1.18], $\sE$ admits a filtration by subbundles whose quotients are unitary flat. If we had $\tilde q(X) = 0$, then $\pi_1(X)$ would
be finite and a suitable \'etale cover $\tilde X$ would be simply connected,
thus the pull-back $\tilde\sE\subset \Omega^1_{\tilde X}$ would be trivial.
But then we have $\smash{H^0(\tilde X,\Omega^1_{\tilde X})}\ne 0$, contradiction.
\end{proof}

\begin{corollary} Assume that $\tilde q(X)=0$ $($or equivalently, that
$\pi_1(X)$ is finite$)$. Then $X$ is projective.
\end{corollary}

\begin{proof} In fact, Theorem 4.4 implies in that case $H^0(X,\Omega^2_X)=0$.
Since $X$ is K\"ahler, we conclude that $X$ is also projective by the Kodaira
embedding theorem.
\end{proof}

\begin{proposition} In addition to the initial hypotheses, assume that $X$
admits a smooth fibration $g:X\to Y$. Then$\;:$
\begin{enumerate}
\smallskip
\item There is a commutative diagram
$$
\xymatrix{A \ar[rr]^{f} \ar[d]^p & & X
   \ar[d]^g
   & \\
  U \ar[rr]^{h} & & Y&}
$$
where $p:A\to U=A/S$ is the quotient map associated with some subtorus $S$
of~$A$ and $h:U\to Y$ is finite.
\smallskip
\item The relative anticanonical bundle $-K_{X/Y}$ is nef.
\smallskip
\item If $X=\bP(E)$ is the projectivization of some vector bundle $E$
of rank $r$ on $Y$, then $E$ is numerically projectively flat, i.e.\
$E\otimes\det(E)^{-1/r}$ is numerically flat $($or equivalently,
$S^rE\otimes\det(E)^{-1}$ is numerically flat$)$.
\end{enumerate}
\end{proposition}

\begin{proof} (1) Since $g\circ f:A\to Y$ is a surjection, Corollary
2.2 shows that it factorizes as $g\circ f=h\circ p$ where $p:A\to
A/S$ is a quotient of $A$ and $h:A/S\to Y$ is a finite map.
\vskip .2cm

\noindent
(2) The diagram induces a generic isomorphism
$$T_{A/U} \to  f^*(T_{X/Y}), $$
so that $f^*(T_{X/Y})$ is generically spanned outside the ramification divisor
$R$ of $f$. Hence $T_{X/Y}$ is nef on all curves $C \not \subset f(R).$
In particular
$$ -K_{X/Y} \cdot C \geq 0$$
for all $C \not \subset f(R).$
Therefore $-K_{X/Y}$ is pseudo-effective by [BDPP04]. We infer from 4.1$\,$(1)
that $-K_{X/Y}$ is actually nef.
\vskip .2cm

\noindent
(3) Assuming that $X=\bP(E)\to Y$, let
$$ \zeta = \sO_{\bP(E)}(1).$$
Since
$$-K_{X/Y} = r\zeta + \phi^*(\det E)^{-1},$$
we conclude that
$$ S^r(E) \otimes \det E^{-1}=S^r\big(E\otimes\det E^{-1/r}\big)$$ is nef.
Since $ c_1(S^r(E) \otimes \det E^{-1}) = 0$, this bundle is numerically flat in the sense of [DPS94], i.e.\ it admits
a filtration whose quotients are hermitian flat (${}={}$unitary
representations of the fundamental group).
\end{proof}

\section{Case of a manifold with finite fundamental group}
\setcounter{lemma}{0}

In this special case, the structure theorem can be stated in a slightly
simpler form.

\begin{proposition} Assume that $X$ is a compact complex manifold
possessing a finite surjective map $f:A\to X$ from a torus. If
$X$ has a finite fundamental group $($or, equivalently, if
$\tilde q(X)=0)$, then
\smallskip
\begin{enumerate}
\item $A$ is an abelian variety and $X\simeq
\bP_{n_1}\times \ldots\times\bP_{n_k}$ is a product of projective spaces.
\smallskip
\item There
is a finite \'etale quotient $\bar A=A/\Gamma$ such that $\bar A$ splits as a product
$\bar A=A_1\times \ldots\times A_k$, and the map $f$ factorizes
through $\bar A$ as a product map $\bar f=f_1\times \ldots\times f_k$ where
$f_j:A_j\to\bP_{n_j}$.
\end{enumerate}
\end{proposition}

\begin{proof}
(1) In fact, we have showed in Corollary 4.5 that $X$ must be projective
algebraic, hence the pull-back of an ample line bundle on $X$ is ample
on $A$ and $A$ is an abelian variety.
By the result of Hwang-Mok [HM01]
already cited in the introduction, $X$ is a tower
$$
X=X_0\to X_1\to\ldots\to X_k=Y
$$
of projective bundles, and the base $Y$ has to be a point, otherwise
the fundamental group would be infinite. By induction on dimension, we
can assume that \hbox{$\phi:X\to X_1$} is a $\bP_{n_1}$-projective
bundle over $X_1$, and since $X_1$ is also covered by the torus $A$
which maps onto $X$, that $X_1=\bP_{n_2}\times
\ldots\times\bP_{n_k}$. Therefore $X=\bP(E)$ for a certain vector
bundle $E\to X_1$. By proposition 4.6$\,$(3), we conclude that
$E$ is projectively numerically flat. However, by Lemma 5.2 below,
this implies that $E=L^{\oplus r}$ for some line bundle $L$ on $X_1$,
$r=n_1+1$, hence $X=\bP_{n_1}\times\ldots\times\bP_{n_k}$.
(Note that, as a consequence, $X$ must in fact be simply
connected$\,$: this is not surprising, since algebraic automorphisms
of $\bP_{n_1}\times \ldots\times\bP_{n_k}$ always have a fixed point
by the Lefschetz fixed point formula, hence $\bP_{n_1}\times
\ldots\times\bP_{n_k}$ cannot possess an \'etale quotient).
\vskip .2cm

\noindent
(2) Let $L_j$ be the pull-back of $\sO(1)$ by the composition $A\to
X\to\bP_{n_j}$ with the $j$-th projection. As $L_j$ is generated by
sections, [BL04, 3.3.2] implies that there is a factorization $A\to
W_j=A/S_j$ by a subtorus $S_j$ and an ample line bundle $G_j$ over
$W_j$ such that $L_j$
is the pull-back of $G_j$ to~$A$. Take $\Gamma=\bigcap S_j$. Then all
our bundles $L_j$ descend to line bundles $\bar L_j$ over $A/\Gamma$ and
therefore, since $\bar L_1+\ldots+\bar L_k$ comes from a very ample
line bundle on $X$, $f$ also factorizes as $\bar f:\bar A=A/\Gamma\to
X$ and $\Gamma$ is finite.  We easily see that $\bar A$ is isomorphic
to the product of its subtori $A_j=\bigcap_{k\ne j}S_k/\Gamma$, and if
$f_j:A_j\to\bP_{n_j}$ is the map induced by the composition
$A_j\subset \bar A\to X\to\bP_{n_j}$, we have $\bar f=f_1\times
\ldots\times f_k$.
\end{proof}

\begin{lemma} Let $Y$ be a product of projective spaces, $E$ a vector
bundle of rank $r$ on $Y$. Suppose  that $E$ is projectively numerically
flat, i.e.\ that $S^r(E) \otimes \det  E^{-1}$ is numerically flat. Then
there exists a line bundle $L$ on $Y$ such that
$$ E \simeq L^{\oplus r}.$$
\end{lemma}

\begin{proof} By a simple argument it suffices to show the lemma for $Y = \bP_m.$ Let $\ell \subset Y$ be a line. Then $E_{|\ell}=\sO(\alpha_1)\oplus
\ldots\oplus\sO(\alpha_r)$ by Grothendieck's theorem, and it is
immediately seen that $E_{|\ell}$ is projectively numerically flat iff
all $\alpha_j$'s are equal to the same integer $\alpha\in\bZ$, where
$\alpha$ is such that $\det E=\sO(r\alpha)$. Hence
$$
(E\otimes\sO(-\alpha))_{|\ell} = \sO^{\oplus r}
$$
for every line $\ell\subset\bP_m$ and we conclude
that $E\otimes\sO(-\alpha) = \sO^{\oplus r}$
(see e.g.\ [OSS80]). This proves our claim in case the base
is $\bP_m$. The case of a product of projective spaces follows.
\end{proof}

\section{The anti-canonical morphism}
\setcounter{lemma}{0}

Let us  consider again a general compact (K\"ahler) manifold $X$ possessing
a finite ramified covering $f:A\to X$ by a torus.

\begin{proposition} The anti-canonical bundle $-K_X$ is semi-ample. Let
$g: X \to Z$ be the associated morphism, which we call the anti-canonical
morphism of $X$. Then $g$ has connected equidimensional fibers and
there is a commutative diagram
$$
\xymatrix{A \ar[rr]^{f} \ar[d]^p & & X
   \ar[d]^g
   & \\
  V \ar[rr]^{h} & & Z&}
$$
where $V=A/S$ is a quotient of $A$ by a subtorus $S$, and $h$ is finite.
Moreover, all fibers of $g$ admit a finite covering by the torus $S$.
\end{proposition}

\begin{proof} If $(-K_X)^n>0$, we know by 4.1$\,$(3) that $-K_X$ is ample, so
we have $Z=X$, $g=\id_X$, and we can take $V=A$, $p=\id_A$.

When $(-K_X)^n = 0$, the ramification divisor $R \subset A$ is nef but not ample
(let us recall that $\sO_A(R)=f^*(-K_X)$). Thus, applying again
[BL04, 3.3.2], we see that there exists a map $p: A \to V$ to a quotient
torus $V$ and an ample divisor $R_V \subset V$ such that
$$
R = p^{-1}(R_V).
$$
In particular, a sufficiently large multiple $\sO_A(kR)$ is spanned, and
the corresponding Kodaira-Iitaka map $\Phi_{|\sO_A(kR)|}$ defines $p.$

Given a nef line bundle $L$, let $\nu(L)$ denote its numerical dimension,
i.e.\ the maximal number $d$ such that $L^d  \not \equiv 0$. In order to
prove that $-K_X$ is semi-ample, it suffices to show that
$$ \kappa(-K_X) = \nu(-K_X),$$
(cf.\ [Kaw85, 6.1] and [Nak87, 5.5]$\;$; more generally, the result holds
true for nef line bundles
$L$ such that $\kappa(L-K_X)=\nu(L-K_X)$, $\nu(aL-K_X)=\nu(L-K_X)$ and
$\kappa(aL-K_X)\ge 0$ for some $a>1$, conditions which are indeed clearly
satisfied for~$L=-K_X$). The above equality $\kappa(-K_X)=\nu(-K_X)$ is
easily verified since
$$
\kappa (-K_X) = \kappa(f^*(-K_X)) = \dim V,
$$
and
$$
\nu(-K_X) = \nu(f^*(-K_X)) = \dim V.
$$
Therefore we obtain an associated morphism $g: X \to Z$ such that
$-mK_X = g^*(L) $ for a fixed suitable number $m$ and a very ample line
bundle $L$ on $Z$. By a general property of Kodaira-Iitaka maps, the
fibers of $g$ are connected if we fix an appropriate
multiple~$m$. Consider such a fiber $X_z$ of $g\,;$ then $f^{-1}(X_z)$
consists of a union of fibers of~$p$, because $p$ is defined
by $f^*(-K_X)$ and its sections are constant along the fibers of
$p:A\to V$. This implies that $g\circ f$ factors through $p$,
and therefore that there is a map $h:V\to Z$ which makes the diagram
commute. We have $\dim Z=\dim V=\kappa(-K_X)$, hence $h$ is a finite
map by Corollary 2.2.  The same result shows that $g\circ f$ has
equidimensional fibers, hence $g$ also has equidimensional fibers,
which must then be images of fibers of $p$ by $f$.
\end{proof}

\begin{proposition} In addition to the notation and hypotheses of
Proposition $6.1$, assume that $q=q(X)=\tilde q(X)$ $($possibly after
replacing $X$ with an \'etale cover$)$. Then
\begin{enumerate}
\item The Albanese map $\alpha:X\to A(X)$ is a smooth surjective fibration
with fibers $F\simeq\bP_{n_1}\times\ldots\times\bP_{n_k}$. Moreover,
$-K_X$ is ample along the fibers $F$ and has numerical dimension
equal to $\dim F= \sum n_j=n-q$.
\smallskip
\item If $R \subset A$ denotes the ramification
divisor of $f$, then $\sO_A(R)=f^*(-K_X)$ is ample along the fibers
of the composition $\alpha\circ f:A\to X\to A(X)$. Furthermore, the map
$$
\Phi=(p,\alpha\circ f):A\to V\times A(X)
$$
induced by the anti-canonical morphism for the first factor~$p$,
and by $\alpha \circ f$ for the second factor, is an isogeny
from $A$ onto $ V\times A(X)$. Moreover $p(f^{-1}f(R))$ is a
proper algebraic subset of $V$.
\end{enumerate}
\end{proposition}

\begin{proof} (1) Theorem 4.2 implies that the fibers $F$ of
$\alpha:X\to A(X)$ have finite fundamental group; as they are also
covered by tori, Proposition 5.1 shows that
$F\simeq\bP_{n_1}\times\ldots\times\bP_{n_k}$. In particular, $F$ is Fano
and so $-K_F$ is ample. This implies that $-K_X$ is ample along the
fibers~$F$, and therefore the numerical dimension $\nu(-K_X)$ is
at least equal to $\dim F=n-q$. If $\nu(-K_X)=\nu(f^*(-K_X))$ was
strictly larger that $n-q$, we would get on the torus $A$
$$
H^{n-q}(A,f^*(K_X))=H^q(A,f^*(-K_X))=0
$$
(see e.g.\ [BL04, 3.4.5]), hence
$$
H^{n-q}(X,K_X)=H^q(X,\sO_X)=0
$$
by taking the direct image. This is absurd.
\vskip .2cm

\noindent
(2) Since $f$ is finite, we see by (1) that $f^*(-K_X)$
is ample along the fibers of~$\alpha\circ f$. On the other hand, as seen
in (6.1), $f^*(-K_X) = \sO_A(R)$ is semi-ample on $A$ and defines the
morphism $p: A \to V$, hence $\sO_A(R)$ is trivial along the fibers of $p$.
Therefore the fibers of $p$ and those of $\alpha\circ f$ can only have
a $0$-dimensional intersection, and this implies that the map
$(p,\alpha\circ f)$ is finite. Since
$$
\dim V=\nu(-f^*K_X)=\nu(-K_X)=n-q=n-\dim A(X),
$$
we conclude that $(p,\alpha\circ f)$ must be an isogeny.  For the last
statement, we notice that the cycle $f_*(R)$ is $\bQ$-linearly
equivalent to $-K_X$ since $f^*(-K_X) = \sO_X(R).$ Thus
$f^*f_*(R)$ is $\bQ$-linearly equivalent to $f^*(-K_X)$. Hence $f^{-1}f(R)$
cannot meet the general fiber of $p: A \to V.$
\end{proof}

\section{Proof of the main theorem}
\setcounter{lemma}{0}

We are now in a position to give all details of the proof of Theorem 1.1.

\begin{proof}
(1) The K\"ahler property follows from Theorem 3.1.
\vskip .2cm

\noindent
(2) First fix an \'etale cover $\tilde X\to X$ such that $q(\tilde
X)=\tilde q(X)$, and a lifting $\tilde f:\tilde A\to\tilde X$. Then by
Proposition 6.2$\,$(2) we get an isogeny
$\Phi=(\tilde p,\tilde \alpha\circ \tilde f):
\tilde A\to \tilde V\times A(\tilde X)$ such that
$\tilde\alpha:\tilde X\to A(\tilde X)$ is the Albanese map of $\tilde X$,
and $\tilde p:\tilde A\to \tilde V$ induces the
anti-canonical image $\tilde Z$ of $\tilde X$ via a commutative diagram
$$
\xymatrix{
   \tilde A \ar[rr]^{\tilde f} \ar[d]^{\tilde p} & & \tilde X
   \ar[d]^{\tilde g} \ar[rr]^{\tilde\alpha} & & A(\tilde X) & \\
  \tilde V \ar[rr]^{\tilde h} & & \tilde Z.&}
$$
On the other hand, we know from 6.2$\,$(1) that $\tilde\alpha$ is a smooth
(locally trivial) fibration and that the fibers $F$ of
$\tilde\alpha$ are products of projective spaces
$\bP_{n_1}\times\ldots\times\bP_{n_k}$. The inverse images $(\tilde f)^{-1}(F)$
are unions of translates of the subtorus $S\subset \tilde A$ equal to
the connected component of $0$ in $\ker(\tilde\alpha\circ\tilde f):
\tilde A\to A(\tilde X)$, a subtorus which is
isogenous to $\tilde V$ via $\tilde p$. Therefore $\tilde g$ maps the fibers
$F$ onto $\tilde Z$, and the restriction $\tilde g_{|F}$ is finite (since
$\tilde h$ is finite and surjective).

Let $\Sigma\subset\tilde Z$ be the union of $\tilde h(\tilde p(\tilde
f^{-1} \tilde f(R)))$ with the set $\tilde Z_\nonet\subset\tilde Z$
above which $\tilde h$ is not \'etale (this includes of course the set
of singular points of $\tilde Z$).  By Proposition 6.2$\,$(2) and the
finiteness of $\tilde h$ this is a proper algebraic subset of $\tilde Z$.
The above argument implies that $\tilde g$ is unramified on
$F\setminus\tilde g^{-1}(\Sigma)$.

We therefore get a ``horizontal direction'' transverse to the fibers of
$F$ by looking
at the fibers of $\tilde g$, and obtain in this way a monodromy of the
fibration $\tilde\alpha$ in terms of a morphism of
$\bZ^{2q}\simeq\pi_1(A(\tilde X))$ into the permutation group of
the finite set $F\cap \tilde g^{-1}(z)$, $z\in\tilde Z\setminus\Sigma$.
The kernel $\Lambda$ of this monodromy map
is a subgroup of finite index in $\pi_1(A(\tilde X))$, and in this way we get
a commutative diagram
$$
\xymatrix{
   X' \ar[rr]^{\alpha'} \ar[d]^u & & A(X') \ar[d]^v & \\
   \tilde X \ar[rr]^{\tilde\alpha} & & A(\tilde X)&}
$$
where the vertical arrows $u$, $v$ are finite \'etale covers and
$v$ is induced by the inclusion $\Lambda\subset\bZ^{2q}$ of the
fundamental groups. Our construction shows that $\alpha'$ is a trivial
fibration [$\,$at least over $X'\setminus u^{-1}(\tilde g^{-1}(\Sigma))$,
but the finiteness of $\tilde g$ implies that the horizontal transport
$$
F_1\cap\big(X'\setminus u^{-1}(\tilde g^{-1}(\Sigma))\big)
\to F_2\cap\big(X'\setminus u^{-1}(\tilde g^{-1}(\Sigma))\big)
$$
between any two fibers $F_1$, $F_2$ must extend to isomorphisms of the
fibers$\,$]. We conclude that
$$
X'=\bP_{n_1}\times\ldots\times\bP_{n_k}\times A(X')
$$
and that the canonical image $Z'$ of $X'$ is precisely
$\bP_{n_1}\times\ldots\times\bP_{n_k}$. Our arguments also imply that
the canonical image $\tilde Z$ (as well as the canonical image $Z$ of the
original manifold $X$) is a finite ramified quotient of $Z'\;$; in fact
any global section of $-mK_X$ pulls back to a global section of
$-mK_{\tilde X}$ or $-mK_{X'}$, and in this way we get naturally
defined maps $Z'\to\tilde Z\to Z$, which are finite by Proposition 6.1
and by obvious commutative diagrams.
\vskip .2cm

\noindent
(3) also follows directly from what we have proved.
\end{proof}

\begin{example} Let us consider $X'=\bP_2\times E_{2\tau}$ where
$E_\tau=\bC/(\bZ+\bZ\tau)$ is the elliptic curve of periods $(1,\tau)$. We take
$X$ to be the finite \'etale quotient of $X'$ by the involution
$(x,t)\mapsto(\sigma(x),t+\tau)$ where $\sigma$ is the involution of
$\bP^2$ given (say) by $x\mapsto -x$ on the affine chart $\bC^2\subset\bP_2$.
In this way we get a fibration $X\to E_\tau$ which is
a locally trivial $\bP_2$-bundle over $E_\tau$  and which is nothing else
than the Albanese map $\alpha:X\to A(X)$. In this case, the anti-canonical
image $Z$ of $X$ is precisely the (singular) quotient
$\bP_2/\langle\sigma\rangle$, because sections of $-mK_X$
pull-back to sections of
$$
H^0(X',-mK_{X'})\simeq H^0(\bP_2,-m K_{\bP_2})
$$
which are invariant by the involution $\sigma$. This simple example shows that
the following two phenomena can occur even under the
assumption $q(\tilde X)=\tilde q(X)$ (here we simply take
$\tilde X=X$)$\,$:
\begin{enumerate}
\item The Albanese map $\tilde X\to A(\tilde X)$ is a non trivial fibration$\,$;
\item The anti-canonical image of $\tilde X$ is singular (and differs from
the product of projective spaces obtained by taking the anti-canonical image of
a suitable \'etale cover~\hbox{$X'\to \tilde X$}).
\end{enumerate}
Of course, it is also easy to produce an example where we additionally have
\vskip .2cm
$\phantom{|}$\kern2pt(3)\kern6pt $q(X')=q(\tilde X)=\tilde q(X)>0=q(X)$.
\vskip .2cm
\noindent
One can take for instance $X$ to be the quotient of
$X'=\bP_2\times (E_{2\tau})^3$
by the finite group${}\simeq\bZ_2^3$ generated by the involutions
\begin{eqnarray*}
&&g_{0}:(x,t_1,t_2,t_3) \mapsto\textstyle
   (\sigma(x),t_1+\frac{1}{2},t_2+{1\over 2},t_3+{1\over 2}),\\
&&g_1:(x,t_1,t_2,t_3) \mapsto
   (x,+t_1+\tau,-t_2+0,-t_3+0),\\
&&g_2:(x,t_1,t_2,t_3) \mapsto
   (x,-t_1+0,+t_2+\tau,-t_3+\tau),\\
&&g_3:(x,t_1,t_2,t_3) \mapsto
   (x,-t_1+\tau,-t_2+\tau,+t_3+\tau),
\end{eqnarray*}
(here $g_3=g_1\circ g_2=g_2\circ g_1$), and let $\tilde X$ be the quotient
of $X'$ by $\langle g_0\rangle\simeq\bZ_2$.
This produces a $\bP_2$-fibration $X\to M$ over an \'etale quotient
of $(E_{2\tau})^3$ with \hbox{$q(M)=0$}, and all phenomena (1), (2), (3)
occur simultaneously.
\end{example}

\section{Appendix: images of K\"ahler spaces by flat morphisms}

We give here a complete proof of Theorem 3.1 under the following
additional assumption that \emph{$X$ has mild singularities}, in the
sense that $X$ is normal and that every point of $X$ admits a neighborhood
$U$ for which there exists a finite ramified cover $\hat U\to U$ which
is smooth (for example, we can take $X$ to have quotient
singularities). In this case, we present below a drastic
simplification of the rather technical arguments of J.~Varouchas
[Var89], which rely extensively on Barlet's theory [Bar75] of cycle
spaces for arbitrary analytic spaces, see also [Cam81], [CP94].
J.~Varouchas probably knew it, but the following simple proof does not
seem to have been published yet.

\begin{proof} Let $f:Y\to X$ be a proper and surjective holomorphic
between complex spaces and let $\omega$ be a K\"ahler metric on~$Y$.
Assuming that the fibers are equidimensional and of pure dimension $p$,
we consider the direct image current
$$
T=f_*(\omega^{p+1}).
\leqno(8.1)
$$
Then clearly $T$ is a K\"ahler current of type $(1,1)$ over $X$ in
the sense of [DP04]. In fact, if $\alpha$ is a smooth positive definite
form on $X$, we have $f^*\alpha\wedge\omega^p\le C_K\omega^{p+1}$ on the
inverse image $f^{-1}(K)$ of any compact set $K\subset X$, therefore
$$
T=f_*(\omega^{p+1})\ge C_K^{-1}f_*(f^*\alpha\wedge\omega^p)
\ge C_K^{-1}(f_*\omega^p)\,\alpha\quad
\hbox{on $K$},
$$
where $f_*\omega^p$ is a (weakly) $d$-closed $(0,0)$ positive current,
namely a collection of positive constants on each of the irreducible
components of $X$. As explained already in section 3, the main point is
to study the continuity of the local potential $u$ of $T$, since it is
then easy to regularize $T$ to obtain a smooth K\"ahler metric.
\vskip.2cm

\noindent
\emph{Step 1.} Assume first that $X$ is non-singular. In this case
we claim$\,$:

\begin{lemma} If $X$ is non-singular, the $(1,1)$ current $T$ defined
by $(8.1)$ admits continuous local potentials.
\end{lemma}

\begin{proof}
By restricting the proper map $f:Y\to X$ over a small neighborhood
of a point $x_0\in X$, we can assume that $X$ is a ball $B(x_0,r)$ in
$\bC^n$. Modulo smooth terms, the local potential $u$ of $T$ is given
by an integral of the form
$$
u(z)=\int_{\zeta\in X}\chi(\zeta)T(\zeta)\wedge\log|z-\zeta|\wedge
dd^c(\log|z-\zeta|)^{n-1},
$$
thanks to the fact that $(dd^c\log|z-\zeta|)^n$ is the current of
integration on the diagonal of $\bC^n\times\bC^n$. Here $\chi$ is a
cut-off function with compact support in~$X$, equal to $1$ on a
neighborhood of~$x_0$, and the integral should be viewed as the direct
image of a current on $\bC^n\times\bC^n$ by the first projection
$(z,\zeta)\mapsto z$. By the definition of $T$ as a direct image,
we can express $u(z)$ as an integral over $Y$, namely a change of
variable $\zeta=f(t)$ yields
$$
u(z)=\int_{t\in Y}\chi(f(t))\omega(t)^{p+1}\wedge\log|z-f(t)|\wedge
dd^c(\log|z-f(t)|)^{n-1}
$$
where $\dim Y=n+p$. The continuity of $u$ follows from
the following more general statement applied to
$v_j(t,z)=\log|z-f(t)|$, thanks  to the fact that all poles of $v_j$ occur
in codimension $n$ on $Y$ by the assumption that $f$ has equidimensional
fibers.
\end{proof}

\begin{lemma} Let $Y,\,Z$ be complex spaces. Consider plurisubharmonic functions
on $Y\times Z$ of the form
$$
v_j(t,z)=\log\sum_k|f_{j,k}(t,z)|^2
$$
where the $f_{j,k}$ are holomorphic functions such that the poles
of the functions\break
$t\mapsto v_j(t,z)$ are at least of codimension $n$
everywhere on $Y$, for every $z\in Z$. Then the wedge products
$$
v_1(\ssbt,z)dd^cv_2(\ssbt,z)\wedge\ldots\wedge dd^c v_\ell(\ssbt,z),
\quad
dd^cv_1(\ssbt,z)\wedge\ldots\wedge dd^c v_\ell(\ssbt,z)
$$
are well-defined currents of locally finite mass whenever $\ell\le n$, and
they depend continuously on $z$ in the weak topology.
\end{lemma}

\begin{proof} Such statements have been known for a long time, see e.g.\
[Dem93, \S$\,$3]. The main point is to get uniform  bounds on the mass
and uniform integrability with respect to the parameter~$z$. Since the result
is local on $Y$, we can assume that $Y$ is a germ of complex space and
use a direct image argument to reduce ourselves to the case of a smooth
variety $Y\,$: just project by a suitable generic projection
$$
\bC^N=\bC^s\times\bC^{N-s}\to \bC^s,\quad s=\dim Y,
$$
from an open set $\Omega\subset\bC^N$ which is a smooth ambient space
for~$Y$. One can use a slicing
argument to reach the situation where the
poles are just isolated points in~$Y$, for every~$z\in Z$ (this actually
amounts to use certain of the coordinates in $Y$ as new parameters). A suitable
application of Stokes theorem and of the comparison principle is
enough to obtain a uniform bound
$$
\Vert dd^cv_2(\ssbt,z)\wedge\ldots\wedge dd^c v_\ell(\ssbt,z)\Vert_{B(t,r)}\le
Cr^2
$$
for the mass on small balls of radius $0<r\le r_0$ (alternatively, this is
a standard estimate
on the Lelong projective masses $\nu(\ssbt, r)$ of our currents -- since
they are not of maximum degree, they are of dimension at least $1$).
The uniform integrability of
$$
v_1(\ssbt,z)dd^cv_2(\ssbt,z)\wedge\ldots\wedge dd^c v_\ell(\ssbt,z)
$$
is finally obtained from an obvious Lojaziewicz type estimate
$$
|v_j(t,z)|\le C|\log d(t,P_j(z))|,
$$
where $P_j(z)\subset Y$ denotes the set of poles of $t\mapsto v_j(t,z)\,$; all
constants $C$ described here can be taken to be locally uniform in $z$.
\end{proof}

\noindent
\emph{Step 2.} Assume now that $X$ is a general normal complex space of 
pure dimension
with \emph{mild singula\-rities}. The main point is to prove the existence and
the continuity of the potential of~$T$. Since this is a local question
on $X$, we may assume
that $X$ admits an irreducible finite ramified covering $\hat X\to X$
such that $\hat X$ is smooth. By taking the fiber product with
$f:Y\to X$, we get a commutative diagram
$$
\xymatrix{
   \hat Y \ar[rr]^{\hat f} \ar[d]^u & & \hat X \ar[d]^v & \\
   Y \ar[rr]^{f} & & X,&}
$$
where  the vertical arrows $u$, $v$ are finite and the horizontal arrows
$f$, $\hat f$ are equidimensional. The fact that $X$ is normal implies
moreover that $\hat Y$ is of pure dimension.
Step 1 shows that the current
$$
\hat T=(\hat f)_*(u^*\omega^{p+1})=v^*(f_*\omega^{p+1})=v^*T
$$
has a continuous potential. Therefore, if $\delta$ is the ramification degree
of $v$, the finite direct image $T= \frac{1}{\delta}v_*v^*T=
\frac{1}{\delta}v_*\hat T$ also has a continuous potential by the
arguments explained in \S$\,$3.
We finally conclude that $T$ can be regularized as a K\"ahler metric
by Richberg's theorem [Ric68] (cf.\ [Dem82, 92]).
\end{proof}

\begin{remark} Without a suitable assumption on the singularities, it
is unclear whether the current $T=f_*(\omega^{p+1})$ admits a local potential,
and even if this potential exists, it need not be continuous. We can take for
instance $Y=\bP^n$, $n\ge 3$, and $f:Y\to X$ equal to the quotient of
$\bP^n$ obtained by identifying two disjoint isomorphic smooth curves
$C_1$, $C_2$ of different degrees, e.g.\ a line and a conic, through a
given isomorphism $C_1\to C_2$. Then
clearly $X$ cannot be K\"ahler since the pull-back of any smooth
closed $(1,1)$-form $\gamma$ on $X$ must have trivial cohomology
class on
$Y$ (the restrictions of $f^*\gamma$ to $C_1$ and $C_2$ are equal but
at the same
time the degrees $f^*\gamma\cdot C_1$ and $f^*\gamma\cdot C_2$ differ
if $f^*\gamma\not\equiv 0$)$\,$; also, in this case, the push forward
$f_*\omega$ of the Fubini-Study K\"ahler form $\omega$
by the quotient map $f:Y\to X$ has a potential which is merely
defined outside $f(C_1)=f(C_2)$ and does not extend continuously
to $X$. As a consequence, the assumption that $X$ is normal is hard
to avoid.

\end{remark}


\begin{thebibliography}{Mum69}

\bibitem[Bar75]{Bar75} Barlet, D.: \emph{Espace analytique r\'eduit des cycles
analytiques complexes compacts d'un espace analytique complexe de
dimension finie},
S\'eminaire F.~Norguet: Fonctions de plusieurs variables complexes,
1974/75,
Lecture Notes in Math., vol.~482,
Springer, Berlin Heidelberg New York,
1975, pp.~1--158.

\bibitem[BL04]{BL04} Birkenhake, C.; Lange, H.:
Complex abelian varieties. 2nd augmented ed.
Grundlehren der Mathematischen Wissenschaften 302. Berlin: Springer, 2004

\bibitem[BDPP04]{BDPP04} Boucksom, S., Demailly, J.-P., P\v{a}un, M.,
Peternell, Th.: \emph{The pseudo-effective cone of a compact K\"ahler
manifold and varieties of negative Kodaira dimension}, math.AG/0405285.

\bibitem[Cam81]{Cam81} Campana, F.: \emph{Cor\'eduction alg\'ebrique d'un espace analytique faiblement k\"ahl\'erien compact}. Invent.\ Math.\  {\bf 63}  (1981), no. 2, 187--223

\bibitem[CP94]{CP94}Campana, F.; Peternell, Th.: \emph{Cycle spaces}. Several complex variables, VII,
319--349, Encyclopaedia Math. Sci., {\bf 74}, Springer, Berlin, 1994

\bibitem[Deb89]{Deb89} Debarre, O.: \emph{Images lisses d'une vari\'et\'e ab\'elienne simple}. C.R.\ Acad.\ Sci.\ Paris {\bf 309} (1989), 119--122

\bibitem[Dem82]{Dem82} Demailly, J.-P.: \emph{Estimations $L^2$ pour
l'op\'erateur $\overline\partial$ d'un fibr\'e
vectoriel holomorphe semi-positif au-dessus d'une vari\'et\'e k\"ahl\'erienne
compl\`ete}. Ann.\ Sci.\ \'Ecole Norm.\ Sup.\ 4e S\'er.\ {\bf 15} (1982)
457--511

\bibitem[Dem92]{Dem92} Demailly, J.-P.: \emph{Regularization of closed
positive currents and intersection theory}. J.\ Algebraic Geometry,
{\bf 1} (1992), 361--409.

\bibitem[Dem93]{Dem93} Demailly, J.-P.:
\emph{Monge-Amp\`ere operators, Lelong numbers and intersection
theory}, Complex Analysis and Geometry, Univ.\ Series in Math.,
edited by V.~Ancona and A.~Silva, Plenum Press, New-York, 1993, 115--193.


\bibitem[DP04]{DP04} Demailly, J.-P.; P\v{a}un, M.: \emph{
Numerical characterization of the K\"ahler cone of a compact K\"ahler manifold}.
Ann.\ Math.\ {\bf 159}, No. 3, (2004), 1247--1274

\bibitem[DPS94]{DPS94} Demailly, J.-P.; Peternell, Th., Schneider, M.: \emph{
Compact complex manifolds with numerically effective tangent bundles}, J.\
Algebraic Geometry {\bf 3} (1994) 295--345

\bibitem[HM01]{HM01} Hwang, J.M.; Mok, N.: \emph{Projective manifolds dominated by abelian varieties}. Math.\ Z.\ {\bf 238} (2001), 89--100

\bibitem[Kaw85]{Kaw85} Kawamata, Y.:  \emph{Pluricanonical systems on minimal algebraic varieties}. Invent.\ Math.\ {\bf 79} (1985), 567--588

\bibitem[Moi67]{Moi67} Moishezon, B. G.: \emph{On $n$-dimensional
compact varieties with $n$ algebraically indepedent meromorphic
functions}. Am. Math. Soc. Transl. II. Ser. 63, (1967) 51-177

\bibitem[Nak87]{Nak87} Nakayama, N.:  \emph{The lower semi-continuity of the plurigenera of complex varieties}. Adv.\ Stud.\ Pure Math.\ {\bf 10} (1987), 551--590

\bibitem[OSS80]{OSS80} Okonek, Ch., Schneider, M., Spindler, H.: \emph{Vector
bundles on complex projective spaces}, Progress in Mathematics, Vol.~3,
Birkh\"auser Boston-Basel-Stuttgart, 1980

\bibitem[Pau98]{Pau98} P\v{a}un, M.: \emph{Sur l'effectivit\'e num\'erique des images inverses de fibr\'es en droites}. Math.\ Ann.\ {\bf 310} (1998), 411--421

\bibitem[Ric68]{Ri68} Richberg, R.: \emph{Stetige streng pseudokonvexe
    Funktionen}.  Math. Ann.\ {\bf 175} (1968) 257--286

\bibitem[Uen75]{Uen75} Ueno, K.: \emph{Classification theory of algebraic
varieties and compact complex spaces}. Lecture Notes in Math.\ {\bf 439},
Springer 1975

\bibitem[Var84]{Var84} Varouchas, J.: \emph{Stabilit\'e de la classe des vari\'et\'es k\"ahl\'eriennes par certains morphismes propres}. Invent.\ Math.\ {\bf 77} (1984),  no. 1, 117--127.

\bibitem[Var89]{Var89} Varouchas, J.: \emph{K\"ahler spaces and proper open morphisms}. Math.\ Ann.\ {\bf 283} (1989), 13--52.

\bigskip
$~$\kern-34pt
(Version of February 22, 2008)

\end{thebibliography}
\end{document}